\documentclass[11pt,reqno]{amsart}%
\usepackage[latin1]{inputenc}
\usepackage{amsfonts,amssymb,latexsym}
\usepackage{amsmath,amsthm}
\usepackage[Sonny]{fncychap}
\usepackage{enumerate}
\usepackage[dvips]{color,graphicx}
\usepackage{amsbsy}
\usepackage{amsfonts}
\usepackage{graphicx}
\usepackage{color}
\usepackage{amsmath}
\usepackage{amssymb}%
\usepackage{amsthm}
\usepackage[dvips]{graphicx}
\usepackage{graphicx} 
\usepackage{subfigure}
\usepackage{pst-eucl}
\usepackage[english]{babel}
\usepackage{amsmath,amssymb,graphics,mathrsfs}
\usepackage{amsmath,amssymb,latexsym}
\usepackage{graphicx,color}
\usepackage[T1]{fontenc}
\usepackage[active]{srcltx}
\usepackage{multicol}
\usepackage[latin1]{inputenc}
\usepackage{pst-all}
\usepackage{enumerate}
\usepackage{pstricks}
\usepackage{pstricks-add}
\usepackage{setspace}
\usepackage{soul}
\usepackage{cancel}
\usepackage{nonfloat}
\usepackage[margin=10pt,font=footnotesize,labelfont=bf,labelsep=endash]{caption}
\usepackage[left=4cm,top=3cm,right=2.4cm,bottom=3.2cm]{geometry}

\usepackage[colorlinks=true,citecolor=red,linkcolor=blue,urlcolor=RubineRed,pdfpagetransition=Blinds,pdftoolbar=false,pdfmenubar=false]{hyperref}

\newtheorem{theorem}{Theorem}[section]
\newtheorem{definition}[theorem]{Definition}

\newtheorem{corollary}[theorem]{Corollary}
\newtheorem{lemma}[theorem]{Lemma}
\newtheorem{remark}{Remark}[section]

\numberwithin{equation}{section}
\begin{document}
\title[ ]{On a bi-dimensional chemo-repulsion model with nonlinear production}
\author[F. Guill\'en-Gonz\'alez]{Francisco Guill\'en-Gonz\'alez}
\address[F. Guill\'en-Gonz\'alez]{Universidad de Sevilla, Dpto. de Ecuaciones Diferenciales y An\'alisis Num\'erico and IMUS, C/Tarfia, S/N, 41012, Sevilla, Spain.}\email{guillen@us.es}
\thanks{FGG has been partially financed by the Projet PGC2018-098308-B-I00, funded by FEDER / Ministerio de Ciencia e Innovación - Agencia Estatal de Investigación.}
\author[E. Mallea-Zepeda]{Exequiel Mallea-Zepeda}
\address[E. Mallea-Zepeda]{Universidad de Tarapac\'a, Departamento de
Matem\'{a}tica, Av. 18 de septiembre, 2222, Arica, Chile.} \email{emallea@uta.cl}

\thanks{E. Mallea-Zepeda was supported by Proyecto UTA-Mayor 4743-19, Universidad de Tarapac\'a.}
%\thanks{E.J. Villamizar-Roa has been supported by VicerrectorÃÂ­a de InvestigaciÃÂ³n y ExtensiÃÂ³n of Universidad Industrial de Santander, and Fondo Nacional de Financiamiento para la Ciencia, la TecnologÃÂ­a y la InnovaciÃÂ³n Francisco JosÃÂ© de Caldas, contrato Colciencias FP 44842-157-2016.}
\author[E.J. Villamizar-Roa]{\'Elder J. Villamizar-Roa}
\address[E.J. Villamizar-Roa]{Universidad Industrial de Santander, Escuela de
Matem\'{a}ticas, A.A. 678, Bucaramanga, Colombia.} \email{jvillami@uis.edu.co}
\thanks{E.J. Villamizar-Roa has been supported by Vicerrector\'{\i}a de Investigaci\'on y Extensi\'on of Universidad Industrial de Santander, proyecto de a\~{n}o sab\'atico.}
\date{\today}

\keywords{Chemo-repuslion, nonlinear production, strong solutions, uniqueness, optimal control.} \subjclass[2010]{35Q35; 49J20; 76D55; 35Q92; 76D05} 

\begin{abstract}
In this paper, we study the following parabolic chemo-repulsion  with nonlinear  production model:
$$
\left\{
\begin{array}{rcl}
\partial_tu-\Delta u&=&\nabla\cdot(u\nabla v),\\
\partial_tv-\Delta v+v&=&u^p+fv\, 1_{\Omega_c}.
\end{array}
\right.
$$
This problem is related to a bilinear control problem, 
where the state $(u,v)$ is the cell density and the chemical concentration respectively, and the control $f$ acts in a bilinear form in the chemical equation. For $2D$ domains, we first consider the case of quadratic signal production ($p=2$), proving the existence and uniqueness of global strong state solution for each control,  and 
the existence of global optimum solution. Afterwards, we deduce the optimality system for any local optimum via a Lagrange  multiplier Theorem, proving regularity of the  Lagrange multipliers.  Finally, we consider the case of signal production $u^p$ with $1<p<2$.
\end{abstract}
\maketitle

\section{Introduction}
\label{intro}
 In last years, the understanding of chemotaxis systems has proceeded to propose more elaborated models in order to capture certain types of mechanisms which are relevant in several situations. A particular case of those mechanisms, which is not properly captured by classical chemotaxis models, corresponds to the process of signal production through cells, which may depend on the population density in a nonlinear manner, as for instance, the saturation effects produced  by some bacterial chemotaxis \cite{Hillen}. See also \cite{Maini,Myerscough,Winkler1}. 

\

In this paper we are interested in the mathematical analysis of a bi-dimensional chemo-repulsion model with nonlinear signal production. By chemo-repulsion we mean the biophysical process of the cell movement towards a lower concentration of chemical substance. This model is given by the following parabolic system in $Q:=\Omega\times (0,T)$:
\begin{equation}\label{eq1}
\left\{
\begin{array}{rcl}
\partial_tu-\Delta u&=&\nabla\cdot(u\nabla v),\\
\partial_tv-\Delta v+v&=&u^p+fv\, 1_{\Omega_c}.
\end{array}
\right.
\end{equation}
 Here, $\Omega\subset\mathbb{R}^2$ is a bounded domain and $(0,T)$ a time interval. The unknowns are $u(t,x)\ge0$ and  $v(t,x)\ge 0$ denoting the cell density and the chemical concentration, respectively. The term $u^p,$ $p>1,$ is the nonlinear production and the reaction term $fv\, 1_{\Omega_c}$ can be interpreted as a bilinear control where the control function $f$ acts as a proliferation or degradation coefficient of the chemical substance, on a subdomain $\Omega_c\subseteq \Omega$. 
 System (\ref{eq1}) is completed with the initial and non-flux boundary conditions
 \begin{equation}\label{eq2}
\left\{
\begin{array}{rcl}
u(0,x)=u_0(x)\ge0,\ v(0,x)=v_0(x)\ge 0\mbox{ in }\Omega,\\
\frac{\partial u}{\partial{\bf n}}=0,\ \frac{\partial v}{\partial{\bf n}}=0\mbox{ on }(0,T)\times\partial\Omega,
\end{array}
\right.
\end{equation}
 where ${\bf n}$ is the outward unit normal vector to $\partial\Omega$.\\

The problem (\ref{eq1})-(\ref{eq2}) but with linear production term and without control ($f \equiv 0$) was studied by Cieslak et al.~in \cite{cieslak}. The authors proved that in 2D bounded domains, this problem has a unique global regular solution, and in spaces of  dimension 3 and 4, one only has  global existence of weak solutions. Even in the linear production case, Tao \cite{tao} considering a smooth bounded convex domain $\Omega$ and a general density-dependent chemotactic sensitivity function, that is, considering in (\ref{eq1}) the term $\nabla \cdot(\chi(u)\nabla v)$ in place of $\nabla \cdot(u\nabla v),$ and $f\equiv 0,$ proved existence (and uniqueness) of global classical  solutions for smooth positive initial data which are  uniformly-in-time bounded. Moreover, it was shown that for any given initial data $(u_0,v_0)$, the corresponding solution $(u,v)$ of (\ref{eq1})-(\ref{eq2}) converges to $(\overline u_0,\overline u_0)$ as time goes to infinity, where  $\overline{u}_0:=\frac{1}{|\Omega|}\int_{\Omega}u_0.$ However,
as far as we know, in the case of chemo-repulsion model with nonlinear signal production (\ref{eq1})-(\ref{eq2}) (in particular, the quadratic production)
the literature is very scarce.  We only known the results of \cite{Diego1,Diego2,Diego3}; in  \cite{Diego1,Diego2} the authors prove the existence of global weak solutions for both two and three dimensions of (\ref{eq1})-(\ref{eq2}) in the quadratic case, with $f\equiv 0$, and global in time strong regularity of the model assuming a regularity criteria, which is satisfied in 2D domains. They also develop some numerical schemes to approximate  weak solutions of (\ref{eq1})-(\ref{eq2}). In \cite{Diego3} the authors proved a result of existence of weak solutions for (\ref{eq1})-(\ref{eq2}) in the superlinear case $1<p<2$ with $f\equiv 0$ and non-negative initial data $(u_0,v_0)\in L^p(\Omega)\times H^1(\Omega).$ That result was obtained as the limit of a sequence of regularized problems in order to deal with the chemotaxis term.  The authors also propose some fully discrete Finite
Element  approximations of problem  (\ref{eq1})-(\ref{eq2}).\\

In addition to the existence and uniqueness of strong solutions $(u,v)$ for the system (\ref{eq1})-(\ref{eq2}), we are interested in the mathematical analysis of an optimal control problem with state equations given by (\ref{eq1})-(\ref{eq2}). We consider a bilinear control problem where the control $f$ acts injecting or extracting chemical substance on a subdomain of control $\Omega_c$.  Controlling the proliferation and death of cells in several environments have important applications in biological processes, as for instance in the formation of bacterial patterns \cite{tTyson,Woodward} or the movement and growth of endothelial cells in response to the  chemical signal known as the tumor angiogenesis factor (TAF), which have an important role in the process of invasion of cancer cells to neighboring tissue \cite{Chaplain,Chaplain1,Mantzaris}. Precisely, we wish to minimize the functional $J$ defined by
\begin{eqnarray}\label{eq1-1}
J(u,v,f)&:=&\frac{\gamma_u}{2}\int_0^T\int_{\Omega_d}|u(t,x)-u_d(t,x)|^2+\frac{\gamma_v}{2}\int_0^T\int_{\Omega_d} |v(t,x)-v_d(t,x)|^2\nonumber\\
&&+\frac{\gamma_f}{2+}\int_0^T\int_{\Omega_c}|f(t,x)|^{2+},
\end{eqnarray}
where $(u,v,f)$ is a strong solution of system (\ref{eq1})-(\ref{eq2}). The functions
$(u_d,\, v_d)\in L^2(Q_d)\times L^2(Q_d)$ are given and represent the desired states with $Q_d=(0,T)\times \Omega_d$ and $\Omega_d\subset \Omega$ the observability subdomain, and the nonnegative reals $\gamma_u,\gamma_v$ and $\gamma_f$ measures
the cost for the states and control, respectively. In (\ref{eq1-1}) the notation $2+$ means $2+\varepsilon,$ for some $\varepsilon>0$ arbitrary. 
Then, the functional $J$ defined in (\ref{eq1-1}) describes the deviation of  the state $(u,v)$ from a desired state $(u_d,v_d)$, plus the cost of the control measured in a given $L^{2+}(\Omega)$-norm.\\

The task of this paper is the following. We first prove the existence and uniqueness of strong solutions $(u,v)$ for the system (\ref{eq1})-(\ref{eq2}) in the quadratic case $p=2,$ for any control $f$, and posteriorly, we  prove the existence of an optimal solution for the related optimal bilinear control problem. Based on a Lagrange multipliers theorem we also obtain the first-order optimality conditions related to any local optimum.
After that,  a regularity result  for the associated Lagrange multipliers is deduced. Finally, we analyze the case of nonlinear signal production $u^p$ for $1<p<2.$ The key point in our analysis is to control the time-derivative of $(\int_\Omega v)^2$ at the same time that the energy, which implies the restriction $p\leq 2.$ Thus, the corresponding existence analysis in the superquadratic case $u^p$ for $p>2$ remains as an  open problem.\\

Not much is known about control problems with state equations given by chemotaxis models (see \cite{dearaujo,fister_mccarthy,rodriguez_rueda,ryu_yagi}).
In \cite{dearaujo} the authors study a distributed optimal control for a two-dimensional  model of cancer invasion; the authors prove the existence of optimal solution and derive an optimality system. In one-dimensional domains, 
two extreme problems  on a chemoattractant  model are analyzed in \cite{fister_mccarthy}; the first one involving harvesting the actual cells and the second one depicts removing a proportion of the chemical substance. 
They prove the existence of optimal solutions and derive an optimality system. 
In \cite{ryu_yagi}, the authors studied an optimal control problem related to the Keller-Segel system to describe the aggregation
process of the cellular slime molds by chemical attraction (see also \cite{Ryu}). In \cite{rodriguez_rueda}, the authors analyze a distributive optimal control problem where the state equations are given
by a stationary chemo-atraction model coupled with the Navier-Stokes equations; the system is controlled through a distributed force and a coefficient of chemotactic sensitivity. The auhors prove existence of optimal solution, and derive some optimality conditions. Recently, in \cite{Exequiel2D,Exequiel3D} the authors studied a bilinear optimal control problem associated to a chemo-repulsion model with linear production term in a bidimensional and a three dimensional bounded domain, respectively. In any case, 
as far as we know, optimal control problems associated with the chemo-repulsion model with nonlinear production term has not been considered in the literature.\\ 

The outline of this paper is as follows:
In Section \ref{sec:2}, we fix the notation, introduce the functional spaces to be used, and give the definition of strong solution for
system (\ref{eq1})-(\ref{eq2}); we also set a Douglas-Niremberg-type inequality which is crucial in our analysis, and 
establish the main results for the quadratic case. In Section
\ref{sec:3}, we prove the existence and uniqueness of strong solution of (\ref{eq1})-(\ref{eq2}).
In Section \ref{sec:4} we analyze the bilinear optimal control problem, including the existence 
of optimal solution, the derivation of first-order optimality conditions, and some extra regularity for the Lagrange multipliers. Finally, in Section \ref{sec:5} we discuss the nonlinear production case $u^p,$ for $1<p<2.$

\section{Preliminaries and main results}
\label{sec:2}

We start establishing some basic notations to be used from now on. Hereafter, $\Omega$ is a bounded domain of $\mathbb{R}^2$ with boundary of class $C^{2,1}.$ We use the Sobolev space $W^{k,q}(\Omega)$ and $L^q(\Omega),$ $k\in \mathbb{R},$ $1\leq q\leq \infty,$ with norms  $\Vert
\cdot\Vert_{W^{k,q}}$ and $\Vert\cdot \Vert_{L^q}$ respectively. When $q=2,$ we write $H^k(\Omega):=W^{k,2}(\Omega)$ and its norm will be denoted by $H^k(\Omega).$ The inner
product in $L^2(\Omega)$ will be represented by $(\cdot,\cdot)$ and the norm by $\|\cdot\|$.  We also consider the space $W_{\bf n}^{m,q}(\Omega)=\{u\in W^{m,q}(\Omega)\,:\, \frac{\partial u}{\partial{\bf n}}=0\mbox{ on }\partial\Omega\}$ ($m>1+1/q$), with norm denoted
by $\|\cdot\|_{W^{m,q}_{\bf n}}$. If $X$ is a Banach space, we denote by $L^q(0,T;X)$ 
the space of valued functions in $X$ defined on the interval
$[0,T]$ that are integrable in the Bochner sense, and its norm will be denoted by $\|\cdot\|_{L^q(X)}$. For  simplicity we denote
$L^q(Q):=L^q(0,T;L^q(\Omega)).$ 
Also, we denote by $C([0,T];X)$ the space of continuous functions from $[0,T]$ into a Banach space
$X$, and its norm by $\|\cdot\|_{C(X)}$.
The  topological dual space of a Banach space $X$ will be denoted by $X'$,
and the duality for a pair $X$ and $X'$ by $\langle\cdot,\cdot\rangle_{X'}$ or simply
by $\langle\cdot,\cdot\rangle$ unless this leads to ambiguity. Moreover, the letters $C$ will denote diverse positive constants which may change from line to line or even within a same line.\\

We will use  the Gagliardo-Nirenberg inequality in $nD$ domains (see \cite{nirenberg}):
\begin{lemma}\label{gagl}
Let $1\le s,t\le+\infty$ such that
$$
\frac1r=\frac{a}{t_*}+\frac{1-a}s,\ a\in[0,1], \quad   \frac1{t_*}=\frac1t-\frac1n.
$$
Then
\begin{equation}\label{gagl-1}
\|u\|_{L^r}\le C_1\|\nabla u\|^a_{L^t}\|u\|^{1-a}_{L^s}+C_2\|u\|_{L^{\tilde{s}}},
\end{equation}
where $C_1$ and $C_2$ are constants which depend on $\Omega$, $t$ and $s$, and
$\tilde{s}>0$ is arbitrary.
\end{lemma}

This Lemma can be seen as the concatenation of some properties; the Sobolev embedding $W^{1,t}$ in $L ^{t_*}$, the interpolation between $L^q$ spaces and the equivalent norm in $W^{1,t}$ given by $\|\nabla u\|_{L^t} + |\int_\Omega u|$.
In particular, we will use Lemma \ref{gagl} for $t=r=n=2$ and $s=1$, hence $a=1/2$, arriving at the inequality
\begin{equation}\label{G-N-2}
\|u\|^2\le C(\|\nabla u\|\|u\|_{L^1}+\|u\|_{L^1}^2),
\end{equation}
and for $r=4$ and $t=n=s=2$, hence $a=1/2$, that means the inequality
\begin{equation}\label{G-N-4}
\|u\|_{L^4}^2\le C(\|\nabla u\|\|u\|+\|u\|^2 ),
\end{equation}
which is a generalized version of the well-known $2D$ interpolation inequality
\begin{equation}\label{2D-4}
\|u\|_{L^4}^2\le C \| u\|_{H^1} \|u\| \quad \forall u\in H^1(\Omega).
\end{equation}

Also, throughout this paper we will use the following equivalent norms in $H^1(\Omega)$ and $H^2(\Omega)$ 
(see \cite{necas} for more details):
\begin{eqnarray}
\|u\|^2_{H^1}&\equiv& \|\nabla u\|^2+\left(\int_\Omega u\right)^2\quad \forall \,u\in H^1(\Omega),\label{norm-1}\\
\|u\|^2_{H^2}&\equiv& \|\Delta u\|^2+\left(\int_\Omega u\right)^2\quad \forall \, u\in H^2_{\bf n}(\Omega),\label{norm-2}
\end{eqnarray}
and the following Banach spaces
\begin{eqnarray*}
X_2&:=&\left\{u\in C([0,T];H^1(\Omega))\cap L^2(0,T;H^2)\,:\, \partial_tu\in L^2(Q)\right\},\\
X_{2+}&:=&\left\{u\in C([0,T];{W}^{1+,2+}(\Omega))\cap L^{2+}(0,T;W^{2,2+})\,:\,\partial_tu\in L^{2+}(Q)\right\}.
%\\
%X_{(4/3)+}&:=&\left\{u\in C([0,T];{W}^{\frac{1}{2}+,\frac{4}{3}+}(\Omega))\cap L^{\frac{4}{3}+}(W^{2,\frac{4}{3}+})\,:\,\partial_tu\in L^{\frac{4}{3}+}(Q)\right\},
\end{eqnarray*}
%where $a+:=a+\varepsilon$ for arbitrary $\varepsilon>0.$

We will study a control problem associated to strong solutions of system (\ref{eq1})-(\ref{eq2}). In the following definition
we give the concept of strong solutions of problem (\ref{eq1})-(\ref{eq2}) in the case of quadratic signal production $p=2.$

\begin{definition}\label{strong-solution}
Let $f\in L^{2+}(Q_c):=L^{2+}(0,T;L^{2+}(\Omega_c))$, $(u_0,v_0)\in H^{1}(\Omega)\times W^{1+,2+}(\Omega)$, with
$u_0\ge0$ and $v_0\ge 0$ a.e. in $\Omega$. A pair $(u,v)$ is called strong solution of system
(\ref{eq1})-(\ref{eq2}) in $(0,T),$ with $p=2$, if $u\ge 0$ and $v\ge 0$ in $Q$,
\begin{equation}
(u,v)\in X_2\times X_{2+},\label{st-1}\\
%v\in S_v&:=&\{v\in L^\infty(H^1(\Omega))\cap L^2(H^2)\,:\, \partial_tv\in L^2(Q)\},\label{st-2}
\end{equation}
the system (\ref{eq1}) is satisfied pointwisely a.e.~$(t,x)\in Q,$
%$$
%\left\{
%\begin{array}{rcl}
%\partial_tu-\Delta u&=&\nabla\cdot(u\nabla v),\\
%\partial_tv-\Delta v+v&=&u^2+fv,
%\end{array}
%\right.
%$$
and the initial and boundary conditions, given in (\ref{eq2}), are veryfied.
\end{definition}
\begin{remark}\label{conser}
The problem (\ref{eq1})-(\ref{eq2}) is conservative in $u$. Indeed, integrating $(\ref{eq1})_1$ in $\Omega$ we have
\begin{equation}\label{conser-1}
\frac{d}{dt}\left(\int_\Omega u\right)=0\, \Rightarrow\, \int_\Omega u(t)=\int_\Omega u_0:= m_0\ \forall t>0.
\end{equation}
Moreover, integrating (\ref{eq1})$_2$ in $\Omega$ we obtain
\begin{equation}\label{conser-2}
\frac{d}{dt}\left(\int_\Omega v\right)+\int_\Omega v=\int_\Omega u^2+\int_{\Omega_c} fv.
\end{equation}
\end{remark}
Our main existence result is given by the following theorem
\begin{theorem}\label{existence}
Under hypothesis of Definition \ref{strong-solution}, 
%Let $f\in L^{2+}(Q_c)$, $(u_0,v_0)\in H^1(\Omega)\times W^{1+,2+}(\Omega)$, with
%$u_0\ge0$ and $v_0\ge$ a.e. in $\Omega$. 
there exists a unique strong solution of system
(\ref{eq1})-(\ref{eq2}), with $p=2,$  in sense of Definition \ref{strong-solution}. Moreover, there exists 
a  constant $$K:=K(T,\|u_0\|_{H^1},\|v_0\|_{W^{1+,2+}},\|f\|_{L^{2+}(Q_c)})>0$$ such that
\begin{equation}\label{bound}
\|(u,v)\|_{X_2\times X_{2+}}\le K.
\end{equation}
\end{theorem}

Now, in order to establish the statement of the bilinear control problem, let  
$\mathcal{F}\subset L^{2+}(Q_c)$ be a nonempty, closed and convex set, 
where $\Omega_c\subset\Omega$ is the control domain, and $\Omega_d\subset\Omega$ is the observability  domain. We assume the data 
$u_0\in H^{1}(\Omega)$, $v_0\in W^{1+,2+}(\Omega)$ with $u_0\ge 0$ and $v_0\ge 0$ in $\Omega$, 
 and  the function $f\in\mathcal{F}$ that describes the bilinear control acting 
on the  $v$-equation.  We consider the  Banach spaces
\begin{equation}\label{opt-1}
X:={X}_2\times{X}_{2+}\times L^{2+}(Q_c),\quad Y:=L^2(Q)\times L^{2+}(Q),
\end{equation}
 the functional $J:X\rightarrow\mathbb{R}$ and the operator $G=(G_1,G_2):X\rightarrow Y$, which at each point 
$(u,v,f)\in X$ are defined by 
\begin{eqnarray}\label{opt-2}
J(u,v,f)&:=&\frac{\gamma_u}{2}\int_0^T\int_{\Omega_d}|u(t,x)-u_d(t,x)|^2+\frac{\gamma_v}{2}\int_0^T\int_{\Omega_d}|v(t,x)-v_d(t,x)|^2\nonumber\\
&&+\frac{\gamma_f}{2+}\int_0^T\int_{\Omega_c}|f(t,x)|^{2+},
\end{eqnarray}
and 
\begin{equation}\label{opt-3}
\left\{
\begin{array}{rcl}
G_1(u,v,f)&=&\partial_tu-\Delta u-\nabla\cdot(u\nabla v),\\
G_2(u,v,f)&=&\partial_tv-\Delta v+v-u^2-fv1_{\Omega_c}.
\end{array}
\right.
\end{equation}
Then, the optimal control problem which we will analyze reads:
\begin{equation}\label{opt-4}
\min_{(u,v,f)\in M}J(u,v,f)\ \mbox{ subject to }\ G(u,v,f)={ 0},
\end{equation}
where 
$$
M:=(\hat{u},\hat{v},\hat{f})+\widehat{X}_2\times\widehat{X}_{2+}\times(\mathcal{F}-\hat{f}),
$$
with $(\hat{u},\hat{v})$ the global strong solution of (\ref{eq1})-(\ref{eq2}) associated to 
$\hat{f}\in\mathcal{F}$ and  
$$
\widehat{X}_2:=\{u\in X_2\,:\, u(0)=0\},\quad \widehat{X}_{2+}:=\{u\in X_{2+}\,:\, u(0)=0\}.
$$
We recall that $u_d,v_d\in L^2(Q_d):=L^2(0,T;L^2(\Omega_d))$ represents the desired states and the
reals $\gamma_u,\gamma_v,\gamma_f$ are nonnegative (not all zero), which measure the cost of the states and control, respectively.\\

 Notice that $M$ is a closed and convex subset of $X\times X\times L^{2+}(Q_c)$. The set of admissible solutions of optimal control problem (\ref{opt-4})
is 
\begin{equation}\label{opt-5}
\mathcal{S}_{ad}=\{x=(u,v,f)\in M\,:\, G(x)={ 0}\}.
\end{equation}

The existence of global optimal solution is given by the following theorem.
\begin{theorem}[Global optimal solution] \label{optimal_solution}
Under hypothesis of Theorem \ref{existence}, if 
%Let  $(u_0,v_0)\in H^1(\Omega)\times W^{1+,2+}(\Omega)$, with
%$u_0\ge0$ and $v_0\ge$ a.e. in $\Omega$. 
we assume that either $\gamma_f>0$ or $\mathcal{F}$
is bounded in  $L^{2+}(Q_c)$, then the optimal control problem (\ref{opt-4}) hast at least one global optimal
solution, that is, there exists $(\tilde{u},\tilde{v},\tilde{f})\in\mathcal{S}_{ad}$ such that
\begin{eqnarray}\label{opsol}
J(\tilde{u},\tilde{v},\tilde{f})=\min_{(u,v,f)\in\mathcal{S}_{ad}}J(u,v,f).
\end{eqnarray}
\end{theorem}
The following result 
guarantees the existence of Lagrange multipliers for the optimal control problem (\ref{opt-4}).
\begin{theorem}[Existence of Lagrange multipliers]\label{Lagrange}
Let $(\tilde{u},\tilde{v},\tilde{f})\in\mathcal{S}_{ad}$ be a local optimal solution for the bilinear control problem (\ref{opt-4}). Then,
there exist Lagrange multipliers $(\lambda,\eta)\in L^2(Q)\times L^{2+}(Q)'$ such that  the following variational inequality holds
\begin{eqnarray}\label{lag-1}
&&\gamma_u\int_0^T\int_{\Omega_d} (\tilde{u}-u_d)U+\gamma_v\int_0^T\int_{\Omega_d}(\tilde{v}-v_d)V
+\gamma_f\int_0^T\int_{\Omega_c}{\rm sgn}(\tilde{f})|\tilde{f}|^{1+}F\nonumber\\
&&+\int_0^T\int_\Omega\Big(\partial_tU-\Delta U-\nabla\cdot(U\nabla\tilde{v})-\nabla\cdot(\tilde{u}\nabla V)\Big)\lambda\nonumber\\
&&+\int_0^T\int_\Omega\left(\partial_tV-\Delta V+V-2\,\tilde{u}\,U-\tilde{f}V1_{\Omega_c}\right)\eta
+\int_0^T\int_{\Omega_c}F\,\tilde{v}\,\eta\ge0,
\end{eqnarray}
for all 
$(U,V,F)\in \widehat{X}_2\times\widehat{X}_{2+}\times\mathcal{C}(\tilde{f})$, 
where 
$\mathcal{C}(\tilde{f}):=\{\theta(f-\tilde{f})\,:\, \theta\ge0,\, f\in\mathcal{F}\}$ is the conical hull  of $\tilde{f}$ in
$\mathcal{F}$.\\
\end{theorem}
In addition to Theorem \ref{Lagrange}, the following theorem provides some extra regularity for the Lagrange multiplier 
$(\lambda,\eta)$ given by Theorem \ref{Lagrange}.
\begin{theorem}\label{regu-lagrange}
The Lagrange multipliers  $(\lambda,\eta)$ given in Theorem \ref{Lagrange} satisfy the following regularity 
\begin{equation}\label{regu-1}
(\lambda,\eta)\in X_2\times X_{(4/3)+},
\end{equation}
where
$$
X_{(4/3)+}:=\left\{u\in C([0,T];{W}^{\frac{1}{2}+,\frac{4}{3}+}(\Omega))\cap L^{\frac{4}{3}+}(W^{2,\frac{4}{3}+})\,:\,\partial_tu\in L^{\frac{4}{3}+}(Q)\right\}.
$$
\end{theorem}
\section{Existence and Uniqueness of Strong Solution}\label{sec:3}
In this section we will prove the existence and uniqueness of strong solutions of system (\ref{eq1})-(\ref{eq2}) (Theorem \ref{existence}). For this,  the Leray-Schauder
fixed-point Theorem will be used. We consider the Banach spaces 
\begin{equation}\label{sys-2}
W_u:=L^{\infty-}(Q)\cap  L^{4-}(W^{1,4-})\ \mbox{ and }\ W_v=L^{\infty-}(Q)
\end{equation}
($L^{\infty-}$ means $L^{q}$ for  $q<\infty$  large enough), 
and the  operator
$$S:W_u\times W_v\rightarrow X_2\times X_{2+} \hookrightarrow W_u\times W_v,$$
where 
$S(\bar{u},\bar{v})=(u,v)$ is the solution of the uncoupled problem (first $v$ can be computed  and afterwards $u$)
\begin{equation}\label{sys-1}
\left\{
\begin{array}{rcl}
\partial_tu-\Delta u&=&\nabla\cdot(\bar{u}_+\nabla v),\\
\partial_tv-\Delta v+v&=&\bar{u}^2+f\bar{v}_+1_{\Omega_c},
%\\
%u(0)&=&u_0,\, v(0)=v_0,\\
%\dfrac{\partial u}{\partial{\bf n}}&=&\dfrac{\partial v}{\partial{\bf n}}=0,
\end{array}
\right.
\end{equation}
endowed with the initial-boundary conditions \eqref{eq2}, 
where $\bar{u}_+:=\max\{\bar{u},0\}\ge0$, $\bar{v}_+:=\max\{\bar{v},0\}\ge0.$ We observe that, into $2D$ domains, the injection of $X_2\times X_{2+}$ in $W_u\times W_v$ is compact. In fact, due to $2D$ inequality \eqref{2D-4} then $X_2\hookrightarrow L^{4}(W^{1,4})$ with continuous embedding.

\

In the following lemmas we will prove the hypotheses of Leray-Schauder fixed-point theorem.
\begin{lemma}\label{compact}
The operator $S$, defined in (\ref{sys-1}) is well defined and completely continuous (compact and continuous).
\end{lemma}
\begin{proof}
Let $(\bar{u},\bar{v})\in W_u\times W_v$. Since $f\in L^{2+}(Q_c)$ and $\bar{v}\in L^{\infty-}(Q)$,
% and taking into account that
%$W_u \hookrightarrow L^{4+}(Q)$ and $W_v\subset L^\infty(Q)$ 
it is easy to deduce that $\bar{u}^2+f\bar{v}_+1_{\Omega_c}\in L^{2+}(Q)$. Then, applying a parabolic regularity result 
in $L^{2+}$ (see \cite[Theorem 10.22, p. 344]{feireisl}),  there exists a unique solution $v\in X_{2+}$ of
(\ref{sys-1})$_2$ and
\begin{equation} \label{sys-3}
\|v\|_{X_{2+}}\le C(\|\overline{u}\|_{L^{4+}(Q)},\|\overline{v}\|_{L^{\infty-}(Q)}\|f\|_{L^{2+}(Q_c)}, \|v_0\|_{W^{1+,2+}}).
\end{equation}
Now, using that $v\in X_{2+}$ and the $2D$ inequality \eqref{2D-4}, one has 
%we have $\nabla v\in L^{2+}(L^{\infty})\cap L^\infty(L^{2+})\hookrightarrow L^{4+}(Q)$. 
$\nabla v\in L^{4+}(Q)$
Therefore, since $(\bar{u}_+,\Delta v)\in L^{\infty-}(Q)\times L^{2+}(Q)$ and $(\nabla\bar{u}_+,\nabla v)\in L^{4-}(Q)\times L^{4+}(Q)$ one can deduce
$$
\nabla\cdot(\bar{u}_+\nabla v)=\bar{u}_+\Delta v+\nabla\bar{u}_+\cdot\nabla v\in L^2(Q).
$$
Then, applying again the parabolic regularity result in $L^2$ \cite[Theorem 10.22, p. 344]{feireisl}, we conclude that there exists a unique
$u\in X_2$ solution of (\ref{sys-1})$_1$ such that
\begin{equation} \label{sys-4}
\|u\|_{X_2} \le C(\|\bar{u}\|_{L^{\infty-}(Q)}\|\Delta v\|_{L^{2+}(Q)},
\|\nabla\bar{u}\|_{L^{4-}(Q)}\|\nabla v\|_{L^{4+}(Q)},
\|u_0\|_{H^1}).
\end{equation}

Therefore, the operator $S$ is well defined. The compactness of $S$ follows of (\ref{sys-3}), (\ref{sys-4}), and the fact that
$X_2\times X_{2+}$ is compactly embedded in $W_u\times W_v$. In particular $S$ maps bounded sets of $W_u\times W_v$ into bounded sets of $X_2\times X_{2+}$. On the other hand,  it is not difficult to prove that $S$ is  continuous.\qed
\end{proof}
\begin{lemma}\label{estimates}
The set of possible fixed-points 
$$S_\alpha:=\{(u,v)\in W_u\times W_v\,:\, (u,v)=\alpha S(u,v),\ \mbox{ for some }\alpha\in[0,1]\}$$
 is bounded in
$W_u\times W_v$. Moreover, for $\alpha\in[0,1]$, there exists a constant 
\begin{equation}\label{es-1}
R:=R(T,\|u_0\|_{H^1},\|v_0\|_{W^{1+,2+}},\|f\|_{L^{2+}(Q_c)})>0
\end{equation}
  ($R$ independent of $\alpha$), such that $\|(u,v)\|_{X_2\times X_{2+}}\le R$  for any $(u,v)\in S_\alpha$. 
\end{lemma}
\begin{proof}
Let $\alpha\in (0,1]$ (the case $\alpha=0$ is trivial). Let consider $(u,v)\in S_\alpha.$ Then, from Lemma \ref{compact},
$(u,v)\in X_2\times X_{2+}$  and satisfies pointwisely a.e. in $Q$ the following problem
\begin{equation}\label{es-2}
\left\{
\begin{array}{rcl}
\partial_tu-\Delta u&=&\nabla\cdot(u_+\nabla v),\\
\partial_tv-\Delta v+v&=&\alpha u^2+\alpha fv_+1_{\Omega_c},
\end{array}
\right.
\end{equation}
endowed with the corresponding initial and boundary conditions. Therefore it is enough to prove that $(u,v)$ is bounded in
$X_2\times X_{2+}$ independent of the parameter $\alpha$. 

\

\underline{Step 1:} $u,v\ge 0$ and $\int_\Omega u(t)=m_0$.
\vspace{0.1cm}

Testing (\ref{es-2})$_1$ by $u_-=\min\{u,0\}\ge0$ and taking
into account that $u_-=0$ if $u\ge0$, $\nabla u_-=\nabla u$ if $u\le 0$ and $\nabla u_-=0$ if $u>0$, we have
$$
\frac12\frac{d}{dt}\|u_-\|^2+\|\nabla u_-\|^2=-(u_+\nabla v,\nabla u_-)=0.
$$
Then, $u_-\equiv0$; so that $u\ge0$. Analogously, testing (\ref{es-2})$_2$ by $v_-$ we deduce
$$
\frac12\frac{d}{dt}\|v_-\|^2+\|v_-\|^2_{H^1}=\alpha(u^2,v_-)+\alpha(fv_+,v_-)_{\Omega_c}\le 0,
$$
which implies $v\ge0$. Therefore, $(u_+,v_+)=(u,v)$. Furthermore, integrating
(\ref{es-2})$_1$ in $\Omega$ we deduce that $\int_\Omega u(t)=m_0$.

\

\underline{Step 2:} $v$ is bounded in $X_2$ and 
$\sqrt{\alpha}u$ in $L^\infty(L^2)\cap L^2(H^1)$. 
\vspace{0.1cm}

Testing (\ref{es-2})$_1$ by $\alpha u$ and (\ref{es-2})$_2$ by $-\frac12\Delta v$, 
integrating by parts and adding the respective equations, chemotaxis and production terms cancel; thus, using the H\"older and Young inequalities and the $2D$ Sobolev embedding $H^1(\Omega)\subset L^{\infty-}(\Omega)$, we get
$$
\frac12\frac{d}{dt}\left(\alpha\|u\|^2+\frac12\|\nabla v\|^2\right)+\alpha\|\nabla u\|^2+\frac12 \|\Delta v\|^2 +\frac12\|\nabla v\|^2
\le \frac{\alpha}{2}\int_\Omega|fv\Delta v|
$$
%$$
%\le \frac14 \|\Delta v\|^2 +C  \,\alpha^2 \|f\|_{L^{2+}}^2 \| v \|_{H^{1}}^2. 
%$$
Then, since $\alpha\le 1$,  applying the H\"older and Young inequalities, we arrive at
\begin{equation}\label{es-3}
\frac{d}{dt}\left( \frac\alpha 2 \|u\|^2+\frac14 \|\nabla v\|^2\right)+\alpha\|\nabla u\|^2+\frac14\|\Delta v\|^2+\frac12\|\nabla v\|^2
\le C  \, \|f\|_{L^{2+}}^2 \| v \|_{H^{1}}^2 .
\end{equation}
On the other hand, integrating (\ref{es-2})$_2$ in $\Omega$, we deduce
\begin{equation}\label{es-5a}
\frac{d}{dt}\left(\int_\Omega v\right)+\int_\Omega v=\alpha\int_\Omega u^2+\alpha\int_{\Omega_c} fv 
\le \alpha \|u\|^2 + \alpha \|f\|\ \| v \|.
\end{equation}
In order to control the first term on the right-side of (\ref{es-5a}), we use the Gagliardo-Nirenberg inequality \eqref{G-N-2} and the fact that $\|u\|_{L^1}=m_0$,
$$
\|u\|^2\le C(\|\nabla u\|\|u\|_{L^1}+\|u\|_{L^1}^2)=C  (\|\nabla u\|+1).
$$
Then, multiplying (\ref{es-5a}) by $\frac12\int_\Omega v$, and applying the Young inequality we get
\begin{eqnarray}\label{es-5}
\frac14\frac{d}{dt} \left(\int_\Omega v\right)^2+\frac14\left(\int_\Omega v\right)^2
%\le \frac{C \alpha m_0}{2} (\| \nabla u\|+1 )\left(\int_\Omega v\right) +C \frac{\alpha}{2} \|f\|\, \|v\|\left(\int_\Omega v\right)\nonumber\\ &&\ \ \ 
 &\le&  \frac{\alpha}{2} \| \nabla u\|^2+C\alpha\left(\int_\Omega v\right)^2\nonumber\\
 &&+C+C \|f\|^2 \|v\|^2.
\end{eqnarray}
Adding (\ref{es-3}) and (\ref{es-5}),  the term depending on $\| \nabla u\|^2$ is absorbed,  obtaining
\begin{equation*}\label{es-7}
\frac{d}{dt}\left(\frac\alpha 2 \|u\|^2+\frac14\|v\|^2_{H^1}\right)+\frac\alpha 2 \|\nabla u\|^2+\frac14 \|v\|^2_{H^2}
\le C(\|f\|^2+\|f\|^2_{L^{2+}}+1)\|v\|_{H^1}^2 + C.
\end{equation*}
Therefore,  the Gronwall Lemma implies 
\begin{equation}\label{cot-1}
\|v\|_{L^\infty(H^1)\cap L^2(H^2)}
+ \|\sqrt{\alpha}u\|_{L^\infty(L^2)\cap L^2(H^1)}
\le K_1(\|u_0\|,\|v_0\|_{H^1},\|f\|_{L^{2+}(Q_c)}).
\end{equation}
Finally, from equation (\ref{es-2})$_2$ and the previous estimates (in particular $v$ is bounded in $L^\infty(L^{\infty -})$ hence $fv1_{\Omega_c}$ is bounded in $L^2(Q)$, and $\sqrt{\alpha}u$ is bounded in $L^4(Q)$ hence $\alpha\,u^2$ is bounded in $L^2(Q)$), one has that $\partial_t v$ is bounded in $L^2(Q)$, hence 
$v$ is bounded in $X_2$.

\

\underline{Step 3:} $u$ is bounded in $L^\infty(L^{\infty-})$. More  precisely, $u^{q/2}$ is bounded in $L^\infty(L^2)\cap L^2(H^1)$ for any $q<\infty$. 
\vspace{0.1cm}

Testing (\ref{es-2})$_1$ by $u^{q-1}$ (for any $q<\infty$) and applying \eqref{2D-4} and Young inequalities we have
\begin{eqnarray}
&&\frac1q\frac{d}{dt}\|u^{q/2}\|^2+\frac {4(q-1)}{q^2}\|\nabla(u^{q/2})\|^2
=
- \frac{2(q-1)}q\int_\Omega u^{q/2}\nabla v \cdot \nabla u^{q/2}\nonumber\\
&&\le \frac{2(q-1)}q \| u^{q/2} \|_{L^4} \| \nabla v \|_{L^4} \| \nabla u^{q/2} \|\nonumber\\
&&\le \frac {2(q-1)}{q^2} \| u^{q/2} \|_{H^1}^2
+  C\frac{(q-1)}{q^2} \| \nabla v \|_{L^4}^4 \|  u^{q/2} \|^2.\nonumber
\end{eqnarray}
Adding $\frac {4(q-1)}{q^2}\|(u^{q/2})\|^2$ to both sides of the previous inequality, we have
\begin{equation}\label{es-10-3}
\frac1q\frac{d}{dt}\|u^{q/2}\|^2+\frac {4(q-1)}{q^2}\| u^{q/2} \|_{H^1}^2
\le \frac{C(q-1)}{q^2} (\| \nabla v \|_{L^4}^4+1) \|  u^{q/2} \|^2.
\end{equation}
Thus, since $\nabla v$ is bounded in $L^4(Q)$ (owing to $v$ is bounded in $X_2$), 
 the Gronwall lemma implies that
\begin{equation}\label{cot-2}
\|u^{q/2}\|_{L^\infty(L^2)\cap L^2(H^1)}
\le K_2(q,\|u_0\|_{L^q},K_1).
\end{equation}

\

\underline{Step 4:} $v$ is bounded in $X_{2+}$.  

\vspace{0.1cm}
 Since $v$ is bounded in $X_2$, from Sobolev embeddings $v$ is also bounded in $L^{\infty-}(Q)$. Then, using that $f\in L^{2+}(Q_c)$ and that $u$ is bounded in $L^\infty(L^{\infty-})$, in particular,  $\alpha u^2+\alpha fv1_{\Omega_c}$ is bounded in $L^{2+}(Q)$. Then, applying the $L^{2+}$ regularity of the Heat-Neumann problem (\cite[Theorem 10.22, p. 344]{feireisl}) 
we deduce that $v\in X_{2+}$ and the following estimate holds
\begin{eqnarray}
\|v\|_{X_{2+}}&\le&\alpha\, C(\|u^2\|_{L^4(Q)}+\|f\|_{L^{2+}(Q_c)}\|v\|_{L^{\infty-}(Q)}+\|v_0\|_{W^{1+,2+}})\nonumber\\
&\le&K_3(\|f\|_{L^{2+}(Q_c)},\|v_0\|_{W^{1+,2+}}, K_2,K_1).\label{cot-2-1}
\end{eqnarray}

\

\underline{Step 5:} $u$ is bounded in $X_{2}$.  
\vspace{0.1cm}

Testing the $u$-equation in (\ref{es-2})$_1$ by $-\Delta u$ we have
\begin{equation}
\frac12\frac{d}{dt}\|\nabla u\|^2+\|\Delta u\|^2
=-(u\Delta v+\nabla u\cdot\nabla v,\Delta u).\label{cot-4}
\end{equation}
Now, applying the H\"older  and Young inequalities, and taking into account the interpolation
inequality   \eqref{2D-4} we have 
\begin{equation}   \label{cot-5}
-(u\Delta v,\Delta u) \le \|u\|_{L^{\infty-}}\|\Delta v\|_{L^{2+}}\|\Delta u\| 
\le\varepsilon\|\Delta u\|^2+C_\varepsilon\|u\|^2_{L^{\infty-}}\|\Delta v\|^2_{L^{2+}},
\end{equation}
\begin{equation}   \label{cot-6}
-(\nabla u\cdot\nabla v,\Delta u)\le\|\nabla u\|_{L^4}\|\nabla v\|_{L^4}\|\Delta u\|
%\le C\|\nabla u\|^{1/2}\|\nabla v\|_{L^4}\|u\|^{3/2}_{H^2}
\le \varepsilon\|u\|^2_{H^2}+C_\varepsilon\|\nabla u\|^2\|\nabla v\|^4_{L^4}.
\end{equation}
Replacing (\ref{cot-5}) and (\ref{cot-6}) in (\ref{cot-4}), choosing $\varepsilon$ small enough, and 
%taking into account that $\left(\int_\Omega u(t)\right)^2=m_0^2$ and the equivalent norms (\ref{norm-1}) %and (\ref{norm-2}), 
adding with the differential inequality \eqref{es-10-3} for $q=2$, we have
\begin{equation}\label{cot-7}
\frac{d}{dt}\|u\|^2_{H^1}+C\|u\|^2_{H^2}\le C\|u\|^2_{L^{\infty-}}\|\Delta v\|^2_{L^{2+}}+C(1+\|\nabla v\|^4_{L^4})\| u\|_{H^1}^2
%+Cm_0^2.
\end{equation}
Then, from the  bounds  of $u$ in $L^\infty(L^{\infty-})$, $\Delta v$ in $L^{2+}(Q)$
and $\nabla v$ in $ L^{4+}(Q)$, the Gronwall lemma implies
\begin{equation}\label{cot-8}
\|u\|_{L^\infty(H^1)\cap L^2(H^2)}\le K_4(\|u_0\|_{H^1},K_3,K_2).
\end{equation}
Therefore, from (\ref{es-2})$_1$, (\ref{cot-2-1}) and (\ref{cot-8}) we deduce
\begin{eqnarray}
\|\partial_tu\|_{L^2(Q)}
%&\le&\|\Delta u+u\Delta v+\nabla u\cdot\nabla v\|_{L^2(Q)}\nonumber\\
&\le&\|\Delta u\|_{L^2(Q)}+\|u\|_{L^{\infty-}(Q)}\|\Delta v\|_{L^{2+}(Q)}+\|\nabla u\|_{L^4(Q)}\|\nabla v\|_{L^{4}(Q)}\nonumber\\
&\le&K_5(K_4,K_3),\label{cot-9}
\end{eqnarray}
which implies that $u$ is bounded in $X_2$.

\

Finally, we conclude that all elements of the set $S_\alpha$ are bounded
in $X_2\times X_{2+}$ by a constant $R$ independently of $\alpha$. The
constant $R$  follows from estimates  (\ref{cot-2-1}), (\ref{cot-8}) and (\ref{cot-9}).\qed
\end{proof}

%\subsection{Proof of Theorem \ref{existence}}
\subsection{Proof of Theorem \ref{existence}}
\begin{proof}

\underline{Existence}

From Lemma \ref{compact} and Lemma \ref{estimates}, we obtain that operator $S$ satisfies  the
conditions of the Leray-Schauder fixed-point theorem. Consequently, there exists a fixed point
$(u,v)=S(u,v)$, which is a solution of problem (\ref{eq1})-(\ref{eq2}) (for $p=2$) in sense of Definition \ref{strong-solution}. The estimate
(\ref{bound}) follows from (\ref{cot-1}) and (\ref{cot-2}).
\vspace{0.2cm}

\underline{Uniqueness}

We consider $(u_1,v_1), (u_2,v_2) \in X_2\times X_{2+}$ two possible  solutions of system
(\ref{eq1})-(\ref{eq2}). Then, subtracting equations (\ref{eq1})-(\ref{eq2}) for $(u_1,v_2)$ and $(u_2,v_2)$, and denoting
$(u,v):=(u_1-u_2,v_1-v_2)$, we  obtain the following problem 
\begin{equation}\label{un-1}
\left\{
\begin{array}{rcl}
\partial_tu-\Delta u&=&\nabla\cdot(u_1\nabla v+u\nabla v_2)\mbox{ in }Q,\\
\partial_tv-\Delta v+v&=&u(u_1+u_2)+fv1_{\Omega_c}\mbox{ in }Q,
%\\
%u(0)&=&v(0)=0\mbox{ in }\Omega,\\
%\dfrac{\partial u}{\partial{\bf n}}&=&\dfrac{\partial v}{\partial{\bf n}}=0\mbox{ on }(0,T)\times\partial\Omega.
\end{array}
\right.
\end{equation}
endowed with the initial-boundary conditions \eqref{eq2}.
Now, testing (\ref{un-1})$_1$ by $u$ and (\ref{un-1})$_2$ by $v-\Delta v$, using that $\int_\Omega u=0$ and the $2D$ inequality  \eqref{2D-4}, and  adding the respective results we have
\begin{eqnarray*}
&&\frac{d}{dt}(\|u\|^2+\|v\|^2_{H^1})+C\| u\|_{H^1}^2+C\|v\|^2_{H^2}\nonumber\\
&&\hspace{0.5cm}\le C(\|u_1+u_2\|^4_{L^4}+\|\nabla v_2\|^4_{L^4})\|u\|^2
%\nonumber\\&&\hspace{0.8cm}
+ C( \|u_1\|^4_{L^4} + \|f\|^2_{L^{2+}})\| v\|_{H^1}^2 .\label{un-9}
\end{eqnarray*}
Thus, since $(u_1,\nabla v_2)\in L^4(Q)\times L^4(Q)$ and $(u_0,v_0)=(0,0)$, then  Gronwall Lemma impies  $u=v=0$, and consequently 
$u_1=u_2$ and $v_1=v_2$.\qed
\end{proof}

\section{The bilinear Optimal Control Problem}\label{sec:4}

In this section we analyze the optimal control problem (\ref{opt-4}) related to System (\ref{eq1})-(\ref{eq2}). We will prove the existence of optimal solution (Theorem \ref{optimal_solution}) and the existence of Lagrange multipliers (Theorem \ref{Lagrange}). In order to get this aim we will use a Lagrange multipliers theorem in Banach spaces (see \cite{Troltzsch} and \cite{Zowe}, for more details). We follow the arguments of \cite{Exequiel2D,Exequiel3D}.
\subsection{Proof of Theorem \ref{optimal_solution}} 
\begin{proof}The proof follows standard arguments. From Theorem \ref{existence} it holds that $\mathcal{S}_{ad}\neq\emptyset$. Let $\{x_ m\}_{m\in\mathbb{N}}=\{(u_m,v_m,f_m)\}_{m\in\mathbb{N}}\subset\mathcal{S}_{ad}$
a minimizing sequence of $J$, that is, $\displaystyle\lim_{m\rightarrow+\infty}J(x_m)=\inf_{x\in\mathcal{S}_{ad}}J(x)$. Then, by definition
of $\mathcal{S}_{ad}$,  for each $m\in\mathbb{N}$, 
$x_m$ satisfies the System
(\ref{eq1})-(\ref{eq2}). Also,
from the definition of $J$ and taking into account the assumption $\gamma_f>0$  or  $\mathcal{F}$ is bounded in $L^{2+}(Q_c)$, it follows that 
\begin{equation}\label{bound_F}
\{f_m\}_{m\in\mathbb{N}}\mbox{ is bounded in }L^{2+}(Q_c).
\end{equation}
Moreover, from  (\ref{bound}) there exists $K>0$, independent of $m$, such that
\begin{equation}\label{bound_u_v}
\|(u_m,v_m)\|_{X_2\times X_{2+}}\le K.
\end{equation}

Therefore, from (\ref{bound_F}), (\ref{bound_u_v}), and taking into account that $\mathcal{F}$ is a closed convex subset of $L^{2+}(Q_c)$
we can deduce the existence of the limit (weak-strong) 
$\tilde{x}=(\tilde{u},\tilde{v},\tilde{f})\in 
 M$
of some subsequence of $\{s_m\}_{m\in\mathbb{N}}$, still denoted by  $\{x_m\}_{m\in\mathbb{N}}$, 
such that $\tilde{x}=(\tilde{u},\tilde{v},\tilde{f})$ is solution of the system pointwisely (\ref{eq1})-(\ref{eq2}), that is, $\tilde{x}\in\mathcal{S}_{ad}$. Here we omit the details. Therefore,
\begin{equation}\label{op20}
\lim_{m\rightarrow+\infty}J(x_m)=\inf_{x\in\mathcal{S}_{ad}}J(x)\le J(\tilde{x}).
\end{equation}
Now, since $J$ is lower semicontinuous on $\mathcal{S}_{ad}$, we get $J(\tilde{x})\le \displaystyle\liminf_{m\rightarrow+\infty} J(x_m)$, which jointly to  
(\ref{op20}), implies (\ref{opsol}).\qed
\end{proof}

\subsection{First-order necessary optimality conditions} 

Concerning to differentiability of the functional $J:X\rightarrow\mathbb{R}$ and operator $G:X\rightarrow Y$ we have the following result.
\begin{lemma}
The functional $J$ is Fr\'echet differentiable and the operator $G$ is continuously Fr\'echet differentiable in
$\tilde{x}=(\tilde{u},\tilde{v},\tilde{f})\in X$, in the direction $s:=(U,V,F)$. The respective derivatives are given by 
\begin{eqnarray}
J'(\tilde{x})[s]&=&\gamma_u\int_0^T\int_{\Omega_d} (\tilde{u}-u_d)U+\gamma_v\int_0^T\int_{\Omega_d}(\tilde{v}-v_d)V\nonumber\\
&&+\gamma_f\int_0^T\int_{\Omega_c}{\rm sgn}(\tilde{f})|\tilde{f}|^{1+}F,\label{derJ}\\
G_1'(\tilde{x})[s]&=&\partial_tU-\Delta U-\nabla\cdot(U\nabla\tilde{v})-\nabla\cdot(\tilde{u}\nabla V),\label{derG-1}\\
G_2'(\tilde{x})[s]&=&\partial_tV-\Delta V+V-2\tilde{u}U-\tilde{f}V1_{\Omega_c}-F\tilde{v}.\label{derG-2}
\end{eqnarray}
\end{lemma}

In order to derive first-order necessary optimality conditions for a local optimal solution $(\tilde{u},\tilde{v},\tilde{f})$ of control problem (\ref{opt-4}), 
we must first prove the existence of Lagrange multipliers. The existence of Lagrange multipliers is guaranteed if a local optimal solution
is a regular point of operator $G$ (see \cite[Theorem 6.3, p. 330]{Troltzsch} or \cite[Theorem 3.1]{Zowe}). 
That is the content of the following result.
\begin{lemma}\label{regular}
Let $\tilde{x}=(\tilde{u},\tilde{v},\tilde{f})\in\mathcal{S}_{ad}$ be a admissible element of bilinear control problem (\ref{opt-4}).
Then, $\tilde{x}$ is regular point, that is, given $(g_u,g_v)\in L^2(Q)\times L^{2+}(Q)$ there exists 
$s=(U,V,F)\in \widehat{X}_2\times\widehat{X}_{2+}\times\mathcal{C}(\tilde{f})$ such that
$$
G'(\tilde{x})[s]=(g_u,g_v),
$$
where $\mathcal{C}(\tilde{f}):=\{\theta(f-\tilde{f})\,:\, \theta\ge0,\, f\in\mathcal{F}\}$ is the conical hull  of $\tilde{f}$ in
$\mathcal{F}$.
\end{lemma}
\begin{proof}
Let $(\tilde{u},\tilde{v},\tilde{f})\in\mathcal{S}_{ad}$ and  
$(g_u,g_v)\in L^2(Q)\times L^{2+}(Q)$. Since $0\in\mathcal{C}(\tilde{f})$ and taking into account \eqref{derG-1} and \eqref{derG-2}, we deduce that it suffices
to prove the existence of $(U,V)\in X_2\times X_{2+}$ solution of the  linear problem 
\begin{equation}\label{reg-1}
\left\{\begin{array}{rcl}
\partial_tU-\Delta U-\nabla\cdot(U\nabla\tilde{v})-\nabla\cdot(\tilde{u}\nabla V)&=&g_u\ \mbox{ in }\ Q,\\
\partial_tV-\Delta V+V-2\,\tilde{u}\,U-\tilde{f}V1_{\Omega_c}&=&g_v\ \mbox{ in }\ Q,\\
U(0)=0,\ V(0)&=&0\ \mbox{ in }\ \Omega,\\
\dfrac{\partial U}{\partial{\bf n}}=0,\ \dfrac{\partial V}{\partial{\bf n}}&=&0\ \mbox{ in }\ (0,T)\times\partial\Omega.
\end{array}
\right.
\end{equation}
In order to prove the existence of solution of system (\ref{reg-1}) we will use the Leray-Schauder fixed-point theorem. Then,
we consider the operator
$$\overline{S}:(\overline{U},\overline{V})\in W_u\times W_v\mapsto (U,V)\in X_2\times X_{2+},
$$ 
with $(U,V)$ the solution of the decoupled problem
\begin{equation}\label{reg-2}
\left\{
\begin{array}{rcl}
\partial_tU-\Delta U-\nabla\cdot(\tilde{u}\nabla V)&=&\nabla\cdot(\bar{U}\nabla\tilde{v})+g_u\ \mbox{ in }\ Q,\\
\partial_tV-\Delta V+V&=&2\tilde{u}\bar{U}+\tilde{f}\bar{V}1_{\Omega_c}+g_v\ \mbox{ in }\ Q,
\end{array}
\right.
\end{equation}
endowed with the  initial and boundary conditions of \eqref{derG-2}. It is not difficult to verify that
operator $\overline{S}$ is well-defined from $W_u\times W_v$ to $X_2\times X_{2+}$, and completely continuous
from $W_u\times W_v$ to itself.

Now, we will prove that the set 
$$
\overline{S}_\alpha:=\{(U,V)\in X_2\times X_{2+}\,:\, (U,V)=\alpha\, \overline{S}(U,V)\mbox{ for some }\alpha \in[0,1] \}
$$
is bounded in $X_2\times X_{2+}$. Indeed, if $(U,V)$ belongs to $\overline{S}_\alpha$, then $(U,V)\in X_2\times X_{2+}$
and solves the linear problem 
\begin{equation}\label{reg-3}
\left\{
\begin{array}{rcl}
\partial_tU-\Delta U-\nabla\cdot(\tilde{u}\nabla V)&=&\alpha\nabla\cdot({U}\nabla\tilde{v})+\alpha g_u\ \mbox{ in }\ Q,\\
\partial_tV-\Delta V+V&=&2\alpha\tilde{u}{U}+\alpha\tilde{f}{V}1_{\Omega_c}+\alpha g_v\ \mbox{ in }\ Q,
\end{array}
\right.
\end{equation}
with the  initial and boundary conditions of \eqref{derG-2}. Testing (\ref{reg-3})$_1$ by $U$ and (\ref{reg-3})$_2$ by
$V-\Delta V$, using the H\"older, Young and 2D interpolation inequality (\ref{2D-4}), and adding the respective results, we can obtain
\begin{eqnarray*}\label{reg-4}
&&\frac{d}{dt}(\|U\|^2+\|V\|^2_{H^1})+C(\|U\|^2_{H^1}+\|V\|^2_{H^2})
%\le C(\alpha^4\|\tilde{u}\|^4_{L^4}+\alpha^2\|\tilde{f}\|^2_{L^{2+}})\|V\|^2_{H^1}
\nonumber\\
%&&+C(1+\alpha^4\|\tilde{u}\|^4_{L^4}+\alpha^2\|\tilde{u}\|^2_{L^4}+\|\nabla\tilde{v}\|^4_{L^4})\|U\|^2 +\alpha^2C(\|g_u\|^2+\|g_v\|^2)\nonumber\\
&&\quad \le C(1+\|\tilde{u}\|^4_{L^4}+\|\tilde{u}\|^2_{L^4}+\|\nabla\tilde{v}\|^4_{L^4})\|U\|^2\nonumber\\
&&\quad +C(\|\tilde{u}\|^4_{L^4}+\|\tilde{f}\|^2_{L^{2+}})\|V\|^2_{H^1}+C(\|g_u\|^2+\|g_v\|^2).
\end{eqnarray*}
From the bounds of $(\tilde{u},\nabla\tilde{v})$ in $L^4(Q)\times L^4(Q)$, the Gronwall lemma implies 
% $(\tilde{f},g_u,g_v)\in L^{2+}(Q_c)\times L^2(Q)\times L^{2+}(Q)$ we deduce that
\begin{equation}\label{reg-5}
\|U\|_{L^\infty(L^2)\cap L^2(H^1)}+\|V\|_{L^\infty(H^1)\cap L^2(H^2)}\le C.
\end{equation}
In particular, one has the bounds of $U$ in $L^4(Q)$ and $V$ in $L^\infty(L^{\infty-})$, hence it is deduced that 
%Now, since $U\in L^\infty(L^2)\cap L^2(H^1)\hookrightarrow L^4(Q)$, $\tilde{u},V\in L^\infty(H^1)\cap L^2(H^2)\subset L^{\infty-}(Q)$,
%and $(\tilde{f},g_v)\in L^{2+}(Q_c)\times L^{2+}(Q)$, we deduce that
$$
2\,\alpha\,\tilde{u}\,U+{\alpha\tilde{f}V1_{\Omega_c}}+\alpha g_v \quad \hbox{is bounded in } L^{2+}(Q).
$$
Then, applying $L^{2+}$-parabolic regularity in (\ref{reg-3})$_2$ (\cite[Theorem 10.22, p. 344]{feireisl}), 
we deduce that $V\in X_{2+}$ and the following estimate holds
\begin{eqnarray} \label{reg-6}
\|V\|_{X_{2+}}\le C.
%\alpha\, C(\|\tilde{u}\|_{L^{\infty-}(Q)}\|U\|_{L^4(Q)}+\|\tilde{f}\|_{L^{2+}(Q_c)}\|V\|_{L^{\infty-}(Q)}+\|g_v\|_{L^{2+}(Q)})\nonumber\\
%&\le&C(T,\|\tilde{u}\|_{L^4(Q)},\|\nabla\tilde{v}\|_{L^4(Q)},\|\tilde{f}\|_{L^{2+}(Q_c)},\|g_v\|_{L^{2+}(Q)}).
\end{eqnarray}
Now, by using a similar argument made to arrive at \eqref{cot-7}, one has
\begin{equation*}\label{cot-7a}
\frac{d}{dt}\|U\|^2_{H^1}+C\|U\|^2_{H^2}\le C\|\tilde u\|^2_{L^{\infty-}}\|\Delta V\|^2_{L^{2+}}
+ C \|\nabla \tilde u\|^4_{L^4}  \|\nabla V\|^4_{L^4} 
%+Cm_0^2.
\end{equation*}
$$
+C(1+\|\nabla \tilde v\|^4_{L^4} + \|\Delta \tilde v\|^2_{L^{2+}})\| U\|_{H^1}^2
+C\|g_u\|^2.
$$
Then,  
%adding (\ref{reg-7}) and (\ref{reg-8}) and 
taking into account 
%(\ref{norm-1}), (\ref{norm-2}), 
estimate (\ref{reg-6}),
the Gronwall Lemma implies that 
% we can deduce that
$U$ is bounded in $L^\infty(H^1)\cap L^2(H^2)$, independent of the parameter $\alpha$. Finally, following similar arguments as in %the Step 5 of the proof of Lemma \ref{estimates}, 
\eqref{cot-9}, we conclude that $\partial_tU$ is bounded in $L^2(Q)$. Consequently,
 $U$ is also bounded in $X_2$.

Therefore, we can deduce the existence of solution for system (\ref{reg-1}) from the Leray-Schauder fixed-point theorem. Moreover, using a classical comparison argument we can prove
that $(U,V)\in X_2\times X_{2+}$ is the unique solution of problem (\ref{reg-1}) for
$(g_u,g_v)\in L^2(Q)\times L^{2+}(Q)$.\qed
\end{proof}

\subsection{Proof of Theorem \ref{Lagrange}}
\begin{proof}
The proof follows directly from Lemma \ref{regular} and  \cite[Theorem 3.1]{Zowe} (see, also \cite[Theorem 6.3, p. 330]{Troltzsch}) and \eqref{derJ}.\qed
\end{proof}
Now, from Theorem \ref{Lagrange} we derive an optimality system for the bilinear optimal control problem (\ref{opt-4}).
\begin{corollary}\label{opt-syst}
Let $\tilde{x}=(\tilde{u},\tilde{v},\tilde{f})$ be a local optimal solution for the optimal control problem
(\ref{opt-4}). Then, the Lagrange multiplier $(\lambda,\eta)\in L^2(Q)\times (L^{2+}(Q))'$ satisfies the system 
\begin{eqnarray}
&&\int_0^T\int_\Omega\big(\partial_tU-\Delta U-\nabla\cdot(U\nabla\tilde{v})\big)\lambda-2\int_0^T\int_\Omega\tilde{u}U\eta\nonumber\\
&&=\gamma_u\int_0^T\int_{\Omega_d}(\tilde{u}-u_d)U\ \ \forall U\in\widehat{X}_2,\label{adj-1}\\
&&\int_0^T\int_\Omega\big(\partial_tV-\Delta V+V\big)\eta
-\int_0^T\int_\Omega\nabla\cdot(\tilde{u}\nabla V)\lambda
-\int_0^T\int_{\Omega_c}\tilde{f}V\eta\nonumber\\
&&=\gamma_v\int_0^T\int_{\Omega_d}(\tilde{v}-v_d)V\ \ \forall V\in\widehat{X}_{2+},\label{adj-2}
\end{eqnarray}
and the optimality condition
\begin{equation}\label{adj-3}
\int_0^T\int_{\Omega_c}\big(\gamma_f{\rm sgn}(\tilde{f})|\tilde{f}|^{1+}+\tilde{v}\eta\big)(f-\tilde{f})\ge 0\ \ \forall{f}\in\mathcal{F}.
\end{equation}
\end{corollary}
\begin{proof}
Taking $(V,F)=(0,0)$ and $(U,F)=(0,0)$ in (\ref{lag-1}), we deduce (\ref{adj-1}) and (\ref{adj-2}), respectively. Similarly,
taking $(U,V)=(0,0)$ in (\ref{lag-1}) we have
$$
\int_0^T\int_{\Omega_c}(\gamma_f{\rm sgn}(\tilde{f})|\tilde{f}|^{1+}+\tilde{v}\eta)F\ge\ \ \forall F\in\mathcal{C}(\tilde{f}).
$$
In particular, choosing $F=f-\tilde{f}\in\mathcal{C}(\tilde{f})$ for any $f\in\mathcal{F}$,  we deduce (\ref{adj-3}).\qed
\end{proof}
\begin{remark}\label{opt-syst1}
System (\ref{adj-1})-(\ref{adj-2}) correspond to the concept of very weak solution of the adjoint system
\begin{equation}\label{adj-pb}
\left\{
\begin{array}{rcl}
-\partial_t\lambda-\Delta\lambda+\nabla\lambda\cdot\nabla\tilde{v}+2\tilde{u}\eta&=&\gamma_u(\tilde{u}-u_d)1_{\Omega_d}\ \mbox{ in }\ Q,\\
-\partial_t\eta-\Delta\eta-\nabla\cdot(\tilde{u}\nabla\lambda)+\eta-\tilde{f}\eta 1_{\Omega_c}&=&\gamma_v(\tilde{v}-v_d)1_{\Omega_d}\ \mbox{ in }\ Q,\\
\lambda(T)&=&0,\ \eta(T)=0\ \mbox{ in }\ \Omega,\\
\dfrac{\partial\lambda}{\partial{\bf n}}&=&0,\ \dfrac{\partial\eta}{\partial{\bf n}}=0\ \mbox{ in }(0,T)\times\partial\Omega.
\end{array}
\right.
\end{equation}
\end{remark}

\subsection{Proof of Theorem \ref{regu-lagrange}}
\begin{proof}
Let $s=T-t$, with $t\in(0,T)$ and $\tilde\lambda(s)=\lambda(t)$, $\tilde\eta(s)=\eta(t)$. Then problem
(\ref{adj-pb}) is equivalent to the system
\begin{equation}\label{regu-2}
\left\{
\begin{array}{rcl}
\partial_s\tilde\lambda-\Delta\tilde\lambda+\nabla\tilde\lambda\cdot\nabla\tilde{v}+2\tilde{u}\tilde\eta&=&\gamma_u(\tilde{u}-u_d)1_{\Omega_d}\ \mbox{ in }\ Q,\\
\partial_s\tilde\eta-\Delta\tilde\eta-\nabla\cdot(\tilde{u}\nabla\tilde\lambda)+\tilde\eta-\tilde{f}\tilde\eta1_{\Omega_c}&=&\gamma_v(\tilde{v}-v_d)1_{\Omega_d}\ \mbox{ in }\ Q,\\
\tilde\lambda(0)&=&0,\ \tilde\eta(0)=0\ \mbox{ in }\ \Omega,\\
\dfrac{\partial\tilde\lambda}{\partial{\bf n}}&=&0,\ \dfrac{\partial\tilde\eta}{\partial{\bf n}}=0\ \mbox{ on }\ (0,T)\times\Omega.
\end{array}
\right.
\end{equation}
Since system (\ref{regu-2}) is linear, we argue in a formal manner proving that any regular enough solution
is bounded in $X_2\times X_{(4/3)+}.$ An exhaustive proof would be done using Leray-Schauder
fixed-point theorem (following the proof of Lemma \ref{regular}). With this aim, testing (\ref{regu-2})$_1$ by $\tilde\lambda-\Delta\tilde\lambda$
and (\ref{regu-2})$_2$ by $\tilde\eta$, using the 2D interpolation inequality (\ref{2D-4}), and adding the respective
results we can obtain
\begin{eqnarray*}
\frac{d}{ds}(\|\tilde\lambda\|^2_{H^1}+\|\tilde\eta\|^2)&+&C(\|\tilde\lambda\|^2_{H^2}+\|\tilde\eta\|^2_{H^1})\nonumber\\
&\le&C(\|\tilde{u}\|^4_{L^4}+\|\nabla\tilde{v}\|^4_{L^4})\|\tilde\lambda\|^2_{H^1}
+C(\|\tilde{u}\|^4_{L^4}+\|\tilde{f}\|^2_{L^{2+}})\|\tilde\eta\|^2\nonumber\\
&&+C\gamma_u^2\|\tilde{u}-u_d\|^2+C\gamma_v^2\|\tilde{v}-v_d\|^2.\label{regu-3}
\end{eqnarray*}
Therefore, since $\tilde{u},\nabla\tilde{v}\in {L^4}(Q)$, the Gronwall Lemma implies
$$
(\tilde\lambda,\tilde\eta)\in L^\infty(H^1\times L^2)\cap L^2(H^2\times H^1).
$$
In particular, $\nabla\tilde\lambda\cdot\nabla\tilde{v}\in L^2(Q)$, hence also $\partial_s\tilde\lambda\in L^2(Q)$, and finally  $\tilde\lambda\in X_2$. Furthermore, using that $\tilde{f}$ belongs to $L^{2+}(Q_c)$ and $\tilde\eta\in L^\infty(L^2)\cap L^2(H^1)\hookrightarrow L^4(Q)$, we get
\begin{equation}\label{regu-4}
\tilde{f}\tilde\eta1_{\Omega_c}\in L^{(4/3)+}(Q).
\end{equation}
 Also, taking into account that
$\tilde{u},\tilde\lambda\in X_2,$ we obtain
\begin{equation}\label{regu-5}
\nabla\cdot(\tilde{u}\nabla\tilde\lambda)=\tilde{u}\Delta\tilde\lambda+\nabla\tilde{u}\cdot\nabla\tilde\lambda
\in L^{2-}(Q).
\end{equation}
Thus, in particular,  $\nabla\cdot(\tilde{u}\nabla\tilde\lambda)$ belongs to $L^{(4/3)+}(Q)$.
Then, from (\ref{regu-4}), (\ref{regu-5}) and applying $L^{(4/3)+}$-regularity in (\ref{regu-2})$_2$ (\cite[Theorem 10.22, p. 344]{feireisl}) we conclude that $\tilde\eta\in X_{(4/3)+}$.
Consequently, we deduce that system (\ref{adj-pb}) has a unique solution $(\lambda,\eta)\in X_2\times  X_{(4/3)+}$.

\

Now we consider $(\lambda,\eta)\in L^2(Q)\times (L^{2+}(Q))'$ the Lagrange multiplier provided by Theorem
 \ref{Lagrange} and $(\hat\lambda,\hat\eta)\in X_2\times  X_{(4/3)+}$ the unique solution of system
 (\ref{adj-pb}). Thus, it suffices to identify $(\lambda,\eta)$ with $(\hat\lambda,\hat\eta)$. For that,
 we consider the unique solution $(U,V)\in X_2\times X_{2+}$ of problem (\ref{reg-1}) for 
 $g_u:=\lambda-\hat\lambda\in L^2(Q)$ and $g_v:={\rm sgn}(\eta-\hat\eta)|\eta-\hat\eta|^{1+}\in L^{2+}(Q)$. Then, writing  problem
 (\ref{adj-pb}) for $(\hat\lambda,\hat\eta)$ instead of $(\lambda,\eta)$, and testing the first equation
 by $U$ and the second equation by $V$, after integrating by parts on $\Omega$, we can obtain 
 
 \begin{eqnarray}
 &&\displaystyle\int_0^T\int_\Omega(\partial_tU-\Delta U-\nabla\cdot(U\nabla\tilde{v}))\hat\lambda
 -2\int_0^T\int_\Omega\tilde{u}U\hat\eta\nonumber\\
 &&=\gamma_u\displaystyle\int_0^T\int_{\Omega_d}(\tilde{u}-u_d)U,\label{regu-6}\\
&&\displaystyle \int_0^T\int_\Omega(\partial_tV-\Delta V+V)\hat\eta-\int_0^T\int_\Omega\nabla\cdot(\tilde{u}\nabla V)\hat\lambda-\int_0^T\int_{\Omega_c}\tilde{f}V\hat\eta\nonumber\\
&&=\gamma_v\displaystyle\int_0^T\int_{\Omega_d}(\tilde{v}-v_d)V.\label{regu-7}
 \end{eqnarray}
Making the difference between (\ref{adj-1}) and (\ref{regu-6}) and between (\ref{adj-2}) and (\ref{regu-7}),
and then summing the respective equalities, we have 
\begin{eqnarray}\label{regu-8}
&&\int_0^T\int_\Omega(\partial_tU-\Delta U-\nabla\cdot(U\nabla\tilde{v}))(\lambda-\hat\lambda)
-2\int_0^T\int_\Omega\tilde{u}U(\eta-\hat\eta)\nonumber\\
&&+\int_0^T\int_\Omega(\partial_tV-\Delta V+V)(\eta-\hat\eta)
-\int_0^T\int_\Omega\nabla\cdot(\tilde{u}\nabla V)(\eta-\hat\eta)\nonumber\\
&&-\int_0^T\int_{\Omega_c}\tilde{f}V(\eta-\hat\eta)=0.
\end{eqnarray}
Thus, taking into account that the pair $(U,V)$ is the unique  solution of (\ref{reg-1}), with data
$g_u=\lambda-\hat\lambda\in L^2(Q)$ and $g_v={\rm sgn}(\eta-\hat\eta)|\eta-\hat\eta|^{1+}\in L^{2+}(Q)$, from (\ref{regu-8}) we deduce that
$$
\|\lambda-\hat\lambda\|^2_{L^2(Q)}+\|\eta-\hat\eta\|^2 _{L^{2}(Q)}=0,
$$
which implies that $\lambda=\hat\lambda$ in $L^2(Q)$ and $\eta=\hat\eta$ in $L^{2+}(Q)$. Therefore, due to the regularity of $(\hat\lambda,\hat\eta)$, we conclude that the Lagrange multipliers 
$(\lambda,\eta)$ satisfy (\ref{regu-1}).\qed
\end{proof}

\section{The nonlinear production case $u^p,$ $1<p<2$}\label{sec:5}
In this section we consider the following chemo-repulsion model with nonlinear production $u^p$ for $1<p<2$:
\begin{equation}\label{eq1n}
\left\{
\begin{array}{rcl}
\partial_tu-\Delta u&=&\nabla\cdot(u\nabla v)\ {\mbox{in} \ Q,}\\
\partial_tv-\Delta v+v&=&u^p+fv\, 1_{\Omega_c}\ { \mbox{in} \ Q,}
%\\
%u(0)&=&u_0,\, v(0)=v_0 \ {in \ \Omega,}
%\\
%\dfrac{\partial u}{\partial{\bf n}}&=&\dfrac{\partial v}{\partial{\bf n}}=0 \ { on \ (0,T)\times\partial\Omega,}
\end{array}
\right.
\end{equation}
endowed with the initial-boundary conditions \eqref{eq2}. 
Then, the aim of this section is to prove the 
%following result of 
existence (and uniqueness) of strong solution of \eqref{eq1n}, as well as to extend  some results of the bilinear optimal control problem 
from  the quadratic case ($p=2$) to the subquadratic case  ($1<p<2$).
\begin{theorem}\label{existence_p}
Let $f\in L^{2+}(Q_c):=L^{2+}(0,T;L^{2+}(\Omega_c))$, $(u_0,v_0)\in H^{1}(\Omega)\times W^{1+,2+}(\Omega)$, with
$u_0\ge0$ and $v_0\ge 0$ a.e. in $\Omega$. Then, there exists a unique strong solution of system
(\ref{eq1n}) in $(0,T)$, that is, there exists a pair $(u,v)$ with $u\ge 0$ and $v\ge 0$ in $Q$ such that
$(u,v)\in X_2\times X_{2+}$ satisfying the system (\ref{eq1n}) pointwisely a.e.~$(t,x)\in Q,$
and the initial and boundary conditions \eqref{eq2}.
\end{theorem}
\begin{proof}
The proof is also carried out through the Leray-Schauder fixerd point Theorem. Let $1<p<2$. We consider the following function spaces
\begin{equation}\label{sys-2n}
W_{u,p}:=L^{\infty-}(Q)\cap  L^{(2p)-}(W^{1,(2p)-})\ \mbox{ and }\ W_v=L^{\infty-}(Q),
\end{equation}
and the  operator
$S_p:W_{u,p}\times W_v\rightarrow X_2\times X_{2+} \hookrightarrow W_u\times W_v$ by
$S_p(\bar{u},\bar{v})=(u,v)$  the solution of the uncoupled problem
\begin{equation}\label{sys-1n}
\left\{
\begin{array}{rcl}
\partial_tu-\Delta u&=&\nabla\cdot(\bar{u}_+\nabla v),\\
\partial_tv-\Delta v+v&=&\bar{u}^p+f\bar{v}_+1_{\Omega_c},
%\\
%u(0)&=&u_0,\, v(0)=v_0,\\
%\dfrac{\partial u}{\partial{\bf n}}&=&\dfrac{\partial v}{\partial{\bf n}}=0.
\end{array}
\right.
\end{equation}
endowed with \eqref{eq2}. 
Notice that, in $2D$ domains, the injection of $X_2\times X_{2+}$ in $W_{u,p}\times W_v$ is compact.
Also, observe that $W_{u,p}\hookrightarrow L^{(2p)+}(Q)$,  
%and $W_v\hookrightarrow L^\infty(Q).$ 
hence $\bar{u}^p\in {L^{2+}(Q)},$ and therefore, following the same arguments of the proof of Lemma {\ref{compact}} we have that the operator $S_{p}$ is well-defined and completely continuous. Now, in order to prove that the set of possible fixed points 
$$S_{\alpha,p}:=\{(u,v)\in W_{u,p}\times W_v\,:\, (u,v)=\alpha S_p(u,v),\ \mbox{ for some }\alpha\in[0,1]\}$$
 is bounded (with respect to $\alpha$) in
$W_{u,p}\times W_v$, we argue as in the proof of Lemma \ref{estimates}.  Let $\alpha\in (0,1]$ (the case $\alpha=0$ is trivial). We consider $(u,v)\in S_{\alpha,p}$, then 
$(u,v)\in X_2\times X_{2+}$  and satisfies pointwisely a.e. in $Q$ the following problem
\begin{equation}\label{es-2n}
\left\{
\begin{array}{rcl}
\partial_tu-\Delta u&=&\nabla\cdot(u_+\nabla v),\\
\partial_tv-\Delta v+v&=&\alpha u^p+\alpha fv_+1_{\Omega_c},
\end{array}
\right.
\end{equation}
endowed with the  initial and boundary conditions \eqref{eq2}. Therefore it suffices to prove that $(u,v)$ is bounded in
$X_2\times X_{2+}$ independent of the parameter $\alpha$. For that, we divide the proof in five steps.

\

\underline{Step 1:} $u,v\ge 0$ and $\int_\Omega u(t)=m_0$. The proof is similar to the proof of Step~1 in Lemma \ref{estimates}.

\

\underline{Step 2:} $v$ is bounded in $X_2$ and 
$\sqrt{\alpha}u^{p/2}$ in $L^\infty(L^2)\cap L^2(H^1)$. 
\vspace{0.1cm}

By testing (\ref{es-2n})$_1$ by $\alpha u^{p-1}$ and (\ref{es-2n})$_2$ by $-\Delta v$; then
integrating by parts, adding the respective equations, chemotaxis and production terms cancel, and using the H\"older and Young inequalities, we have
%$$
%\frac{d}{dt}\left(\frac{\alpha}{p-1}\|u^{p/2}\|^2+\frac 1 2\|\nabla v\|^2\right)+\frac{4\alpha}{p}\|\nabla (u^{p/2})\|^2+\|\Delta v\|^2 +\|\nabla v\|^2
%\le \frac{\alpha}{2}\int_\Omega|fv\Delta v|
%$$
%$$
%\le \frac12 \|\Delta v\|^2 +C  \frac{\alpha^2}2 \|f\|_{L^{2+}}^2 \| v \|_{H^{1}}^2. 
%$$
%Thus we arrive at
\begin{eqnarray}\label{es-3n}
&&\frac{d}{dt}\left(\frac{\alpha}{p-1}\|u^{p/2}\|^2+\frac12\|\nabla v\|^2\right)+\frac{4\alpha}{p}\|\nabla (u^{p/2})\|^2+\frac 12\|\Delta v\|^2 +\|\nabla v\|^2\nonumber\\
&&\le C \, \|f\|_{L^{2+}}^2 \| v \|_{H^{1}}^2 .
\end{eqnarray}
On the other hand, integrating (\ref{es-2n})$_2$ in $\Omega$, we have
\begin{equation}\label{xyzn}
\frac{d}{dt}\left(\int_\Omega v\right)+\int_\Omega v=\alpha\int_\Omega u^p+\alpha\int_{\Omega_c} fv.
\end{equation}
Then, multiplying (\ref{xyzn}) by $\int_\Omega v$ we deduce that
\begin{equation}\label{es-5an}
\frac 12\frac{d}{dt}\left(\int_\Omega v\right)^2+\frac 12\left(\int_\Omega v\right)^2\leq {\alpha^2}\Vert u\Vert^{2p}_{L^p}+C\,\alpha^2 \Vert f\Vert^2_{L^{2+}}\Vert v\Vert_{H^1}^2.
\end{equation}
In this point, since $1<p<2,$ we use the Gagliardo-Nirenberg inequality (\ref{gagl-1}) and Young inequality in order to bound:
\begin{eqnarray}
{\alpha^2}\Vert u\Vert^{2p}_{L^p}&=&{\alpha^2}\Vert u^{p/2}\Vert^4 
\leq C\,\alpha^2\Vert \nabla (u^{p/2})\Vert^{4(p-1)/p}\Vert u^{p/2}\Vert_{L^{2/p}}^{4/p}
+C\,\alpha^2\Vert u^{p/2}\Vert^4_{L^{2/p}}\nonumber\\
&\leq& \frac{2\alpha}{p}\Vert \nabla (u^{p/2})\Vert^2+C\alpha^{2p/(2-p)}\Vert u\Vert_{L^1}^{4/(2-p)}+C\alpha^2\Vert u\Vert_{L^1}^{2p}.\label{zzn}
\end{eqnarray}
Therefore, adding (\ref{es-3n})-(\ref{es-5an}) and using (\ref{zzn}) and that $\Vert u\Vert_{L^1}=m_0$,  we have
\begin{eqnarray}\label{es-3zn}
&&\frac{d}{dt}\left(\frac{\alpha}{p-1}\|u^{p/2}\|^2+\frac12\|v\|_{H^1}^2\right)+\frac{2\alpha}{p}\|\nabla (u^{p/2})\|^2+\frac 12\| v\|_{H^2}^2\nonumber\\
&&\leq C \, \|f\|_{L^{2+}}^2 \| v \|_{H^{1}}^2 + C.
\end{eqnarray}
Therefore, from (\ref{es-3zn}) and Gronwall lemma we deduce
\begin{equation}\label{cot-1n}
\|v\|_{L^\infty(H^1)\cap L^2(H^2)}
+ \|\sqrt{\alpha}u^{p/2}\|_{L^\infty(L^2)\cap L^2(H^1)}
\le K_1(\|u_0\|_{L^p},\|v_0\|_{H^1},\|f\|_{L^{2+}(Q_c)}).
\end{equation}
Finally, from equation (\ref{es-2n})$_2$ and the previous estimates (in particular $v$ is bounded in $L^\infty(L^{\infty -})$ it holds that $f\, v1_{\Omega_c}$ is bounded in $L^2(Q),$ and $\sqrt{\alpha} u^{p/2}$ is bounded in $L^4(Q),$ which implies that $\alpha u^p$ is bounded in $L^{2}(Q)$. Therefore from \eqref{es-2n}$_2$ it holds that $\partial_t v$ is bounded in $L^2(Q)$, and thus, 
$v$ is bounded in $X_2$.

\

\underline{Step 3:} $u$ is bounded in $L^\infty(L^{\infty-})$ and $u^{q/2}$ in $L^2(H^1)$ for any $q<\infty$. 
\vspace{0.1cm}
The proof is similar to the proof of Step 3 in Lemma \ref{estimates}.

%Testing (\ref{es-2n})$_1$ by $u^{q-1}$ (for any $q<\infty$) and applying 2D interpolation inequality \eqref{2D-4} we have
%\begin{eqnarray*}
%&&\frac1q\frac{d}{dt}\|u^{q/2}\|^2+\frac {4(q-1)}{q^2}\|\nabla(u^{q/2})\|^2
%=
%- \frac{2(q-1)}q\int_\Omega u^{q/2}\nabla v \cdot \nabla u^{q/2}\\
%&&\le \frac{2(q-1)}q \| u^{q/2} \|_{L^4} \| \nabla v \|_{L^4} \| \nabla u^{q/2} \|\\
%&&\le \frac{2(q-1)}{q^2} \| u^{q/2} \|_{H^1}^2
%+  C\, q^{2}(q-1) \| \nabla v \|_{L^4}^4 \|  u^{q/2} \|^2.
%\end{eqnarray*}
%Adding $\frac {2(q-1)}{q^2}\|u^{q/2}\|^2$ to both sides of the previous inequality we have
%\begin{equation}\label{es-10-3n}
%\frac1q\frac{d}{dt}\|u^{q/2}\|^2+\frac {2(q-1)}{q^2}\| u^{q/2} \|_{H^1}^2
%\le  C(q)\, \|  u^{q/2} \|^2(\| \nabla v \|_{L^4}^4 +1).
%\end{equation}
%Thus, applying Gronwall lemma to (\ref{es-10-3n}), and taking into account that $\nabla v $ is bounded in $ L^4(Q)$, we deduce that 
%\begin{equation}\label{cot-2n}
%\|u^{q/2}\|_{L^\infty(L^2)\cap L^2(H^1)}
%\le K_2(q,\|u_0\|_{L^q},K_1).
%\end{equation}

\

{\underline{Step 4:} $(u,v)$ is bounded in $X_{2}\times X_{2+}$. The proof is similar to the proof of Steps 4 and 5 in Lemma \ref{estimates}.
\vspace{0.1cm}
Thus, we conclude that all elements of the set $S_{\alpha,p}$ are bounded
in $X_2\times X_{2+}$ by a constant $R$ independently of $\alpha$.} \qed
\end{proof}
 In order to establish the optimal control results related to the superlinear (and sub-quadratic) case $1<p<2$, we consider the same  Banach spaces $X:={X}_2\times{X}_{2+}\times L^{2+}(Q_c)$ and $Y:=L^2(Q)\times L^{2+}(Q)$, and the same cost functional $J:X\rightarrow\mathbb{R}$ defined in (\ref{opt-2}). The state operator 
$G_p=(\widetilde{G}_1,\widetilde{G}_2):X\rightarrow Y$, which at each point 
$(u,v,f)\in X$ is now defined by 
\begin{equation}\label{opt-3p}
\left\{
\begin{array}{rcl}
\widetilde{G}_1(u,v,f)&=&\partial_tu-\Delta u-\nabla\cdot(u\nabla v),\\
\widetilde{G}_2(u,v,f)&=&\partial_tv-\Delta v+v-u^p-fv1_{\Omega_c},
\end{array}
\right.
\end{equation}
with $1<p<2.$ Then, the optimal control problem for the superlinear case reads:
\begin{equation}\label{opt-4p}
\min_{(u,v,f)\in M}J(u,v,f)\ \mbox{ subject to }\ G_p(u,v,f)={ 0},
\end{equation}
where 
$$
M:=(\hat{u},\hat{v},\hat{f})+\widehat{X}_2\times\widehat{X}_{2+}\times(\mathcal{F}-\hat{f}),
$$
with $(\hat{u},\hat{v})$ the global strong solution of (\ref{eq1})-(\ref{eq2}) associated to a 
$\hat{f}\in\mathcal{F}$. 
%$1<p<2,$ and  
%$$
%\widehat{X}_2:=\{u\in X_2\,:\, u(0)=0\},\quad \widehat{X}_{2+}:=\{u\in X_{2+}\,:\, u(0)=0\}.
%$$
Then, the existence of optimal solution is given by the following theorem. Its proof is similar to the proof of Theorem \ref{optimal_solution}; therefore, we omit it.
\begin{theorem}[Global optimal solution] \label{optimal_solution_p}
Under hypothesis of Theorem \ref{existence_p}, if 
%Let  $(u_0,v_0)\in H^1(\Omega)\times W^{1+,2+}(\Omega)$, with
%$u_0\ge0$ and $v_0\ge$ a.e. in $\Omega$. 
we assume that either $\gamma_f>0$ or $\mathcal{F}$
is bounded in  $L^{2+}(Q_c)$, then the optimal control problem (\ref{opt-4p}) hast at least one global optimal
solution, that is, there exists $(\tilde{u},\tilde{v},\tilde{f})\in\mathcal{S}_{ad}$ such that
\begin{eqnarray}\label{opsol_p}
J(\tilde{u},\tilde{v},\tilde{f})=\min_{(u,v,f)\in\mathcal{S}_{ad}}J(u,v,f).
\end{eqnarray}
\end{theorem}
The following result 
guarantees the existence and regularity of Lagrange multipliers for the optimal control problem (\ref{opt-4p}). Its proof is a slight modification of the proof of Theorems \ref{Lagrange} and \ref{regu-lagrange}; therefore we omit it.
\begin{theorem}[Existence of Lagrange multipliers]\label{Lagrangep}
Let $(\tilde{u},\tilde{v},\tilde{f})\in\mathcal{S}_{ad}$ be a local optimal solution for the bilinear control problem (\ref{opt-4p}). Then,
there exist Lagrange multipliers $(\lambda,\eta)\in L^2(Q)\times (L^{2+}(Q))'$ such that  the following variational inequality holds
\begin{eqnarray}\label{lag-1}
&&\gamma_u\int_0^T\int_{\Omega_d} (\tilde{u}-u_d)U+\gamma_v\int_0^T\int_{\Omega_d}(\tilde{v}-v_d)V
+\gamma_f\int_0^T\int_{\Omega_c}{\rm sgn}(\tilde{f})|\tilde{f}|^{1+}F\nonumber\\
&&+\int_0^T\int_\Omega\left(\partial_tU-\Delta U-\nabla\cdot(U\nabla\tilde{v})-\nabla\cdot(\tilde{u}\nabla V)\right)\lambda\nonumber\\
&&+\int_0^T\int_\Omega\left(\partial_tV-\Delta V+V-p\tilde{u}^{p-1}U-\tilde{f}V1_{\Omega_c}\right)\eta
+\int_0^T\int_{\Omega_c}F\tilde{v}\eta\ge0,
\end{eqnarray}
for all 
$(U,V,F)\in \widehat{X}_2\times\widehat{X}_{2+}\times\mathcal{C}(\tilde{f})$, 
where 
$\mathcal{C}(\tilde{f}):=\{\theta(f-\tilde{f})\,:\, \theta\ge0,\, f\in\mathcal{F}\}$ is the conical hull  of $\tilde{f}$ in
$\mathcal{F}$.
In addition, the Lagrange multipliers  $(\lambda,\eta)$
have the following regularity 
\begin{equation}\label{regu-1p}
(\lambda,\eta)\in X_2\times X_{(4/3)+},
\end{equation}
and satisfy the adjoint system
\begin{equation}\label{adj-superlinear}
\left\{
\begin{array}{rcl}
-\partial_t\lambda-\Delta\lambda+\nabla\lambda\cdot\nabla\tilde{v}
+p\tilde{u}^{p-1}\eta&=&\gamma_u(\tilde{u}-u_d)1_{\Omega_d}\ \mbox{ in }\ Q,\\
-\partial_t\eta-\Delta\eta-\nabla\cdot(\tilde{u}\nabla\lambda)+\eta-\tilde{f}\eta 1_{\Omega_c}&=&\gamma_v(\tilde{v}-v_d)1_{\Omega_d}\ \mbox{ in }\ Q,\\
\lambda(T)&=&0,\ \eta(T)=0\ \mbox{ in }\ \Omega,\\
\dfrac{\partial\lambda}{\partial{\bf n}}&=&0,\ \dfrac{\partial\eta}{\partial{\bf n}}=0\ \mbox{ in }(0,T)\times\partial\Omega.
\end{array}
\right.
\end{equation}
\end{theorem}

%\begin{acknowledgements}
%F. Guill\'en-Gonz\'alez has been partially financed by the Project PGC2018-098308-B-I00, funded by FEDER / Ministerio de Ciencia e Innovaci\'on - Agencia Estatal de Investigaci\'on.
%E. Mallea-Zepeda has been  supported by Proyecto UTA-Mayor 4743-19, Universidad de Tarapac\'a.
%E.J. Villamizar-Roa has been supported by Vicerrector\'{\i}a de Investigaci\'on y Extensi\'on of Universidad %Industrial de Santander, proyecto de a\~{n}o sab\'atico.
%\end{acknowledgements}

\end{document}